\documentclass[a4paper]{amsart}

\usepackage[foot]{amsaddr}
\usepackage[utf8]{inputenc}
\usepackage{amsmath}
\usepackage{amsfonts}
\usepackage{amssymb}
\usepackage{amsthm}
\usepackage{color}
\usepackage{tikz}
\usepackage{mathtools}
\mathtoolsset{showonlyrefs,showmanualtags} %scommentare per far numerare solo le equationi che sono chiamate nel testo

\newtheorem{theorem}{Theorem}[section]
\newtheorem*{theorem*}{Theorem}
\newtheorem{corollary}[theorem]{Corollary}
\newtheorem{lemma}[theorem]{Lemma}
\newtheorem*{lemma*}{Lemma}

\newtheorem*{proposition*}{Proposition}

\theoremstyle{definition} %% COSI DEF E REMARK NON SONO SCRITTI IN ITALIC
\newtheorem{definition}[theorem]{Definition}
\newtheorem*{definition*}{Definition}

\theoremstyle{remark}
\newtheorem{remark}[theorem]{Remark}

\newcommand{\bi}{\begin{itemize}}
\newcommand{\iii}{\item}
\newcommand{\ei}{\end{itemize}}
\newcommand{\bd}{\begin{description}}
\newcommand{\ed}{\end{description}}
\newcommand{\bdeff}{\begin{definition}}
\newcommand{\edeff}{\end{definition}}

\newcommand{\bqn}{\begin{eqnarray}}
\newcommand{\eqn}{\end{eqnarray}}

\newcommand{\R}{\mathbb{R}}						% reals
\newcommand{\bbN}{{\mathbb{N}}}
\newcommand{\bbR}{{\mathbb{R}}}
\newcommand{\tr}{{\mathrm {tr}}}
\newcommand{\rk}{{\mathrm {rk}}}
\newcommand{\calF}{{\mathcal{F}}}
\newcommand{\calG}{{\mathcal{G}}}

\newcommand{\calI}{{\mathcal{I}}}
\newcommand{\calN}{{\mathcal{N}}}
\newcommand{\calQ}{{\mathcal{Q}}}

\newcommand{\calX}{{\mathcal{X}}}
\newcommand{\calY}{{\mathcal{Y}}}

\author{Davide Barilari$^\flat$}
\address{$^\flat$Institut de Math\'ematiques de Jussieu-Paris Rive Gauche UMR CNRS 7586, Universit\'e Paris-Diderot,
Batiment Sophie Germain, Case 7012, 75205 Paris Cedex 13, France} \email{davide.barilari@imj-prg.fr}
\author{Elisa Paoli$^\sharp$}
\address{$^\sharp$SISSA, Via Bonomea 265, Trieste, Italy}
\email{epaoli@sissa.it}
\date{\today}

\title[Curvature in heat kernel expansion for  H\"{o}rmander operators]{Curvature terms in small time heat kernel expansion for a model class of hypoelliptic H\"{o}rmander operators}

\begin{document}
	
\begin{abstract}
We consider the heat equation associated with a class of second order hypoelliptic H\"{o}rmander operators with constant second order term and linear drift. We describe the possible small time heat kernel expansion on the diagonal giving a geometric characterization of the coefficients in terms of the divergence of the drift field and the curvature-like invariants of the optimal control problem associated with the diffusion operator. 
\end{abstract}

\maketitle

\tableofcontents

\newcommand{\N}{\mathbb{N}}
\newcommand{\cb}[1]{\textcolor{blue}{#1}}
\section{Introduction}
The interaction between analysis and geometry in the study of the heat equation in several smooth (and more recently, non-smooth) contexts has attracted an increasing attention during the last decades. 

On a Riemannian manifold, many results have been obtained relating the small time asymptotics of the fundamental solution of the heat equation, the so-called \emph{heat kernel}, to geometric quantities such as distance \cite{varadhan}, cut and conjugate locus \cite{molchanov,neelstroock}, curvature invariants \cite{spectrebook}. The extension of these results to non-Riemannian situations (from the geometric viewpoint) or non-elliptic operators (from the viewpoint of PDE), when possible, is non trivial: some results have been obtained relating the hypoelliptic heat kernel with its associated Carnot-Carath\'eodory distance \cite{leandremaj,leandremin} and its cut locus \cite{BBCN13,barbosneel}, but much less is known concerning the relation with curvature. This is also related to the difficulty of defining a general notion of curvature associated to Carnot-Carath\'eodory metrics.  %As a non exaustive list one  \cite{citare varie cose}

One of the most celebrated results in the Riemannian setting, which is by now classical and in the spirit of the investigation presented in this paper, reads as follows. Denote by $p(t,x,y)$  the heat kernel associated with the Laplace-Beltrami operator $\Delta_{g}$  on a Riemannian manifold $(M,g)$, then the coefficients appearing in the small time heat kernel expansion on the diagonal 
\begin{equation}\label{eq1}
p(t,x_{0},x_{0})=\frac{1}{t^{n/2}}\left(\sum_{i=0}^{m}a_{i}(x_{0}) t^{i}+O(t^{m})\right)
\end{equation}
 contain information about the curvature of the manifold at the point $x_{0}$, i.e., all $a_{i}(x_{0})$
can be written as universal polynomials of the Riemann tensor and its covariant derivatives computed at the point $x_{0}$ (see for instance \cite{spectrebook,book:RosenbergLaplacian}). 

More generally, if one is interested in the study of the heat kernel associated with a smooth second order elliptic operator $L$, possibly with drift, one can extract the good Riemannian metric $g$ from the principal symbol of the operator $L$, such that the on-the-diagonal heat kernel expansion contains geometric information about curvature associated to the metric $g$ and the local structure of the drift at the point. For instance, in \cite{bismut}, it is proved that for the operator $L=\Delta_{g}+X_{0}$ associated to a Riemannian manifold $(M,g)$ the corresponding heat kernel satisfies for $t\to 0$
\begin{equation}\label{eq:2}
p(t,x_{0},x_{0})=\frac{1}{(4\pi t)^{n/2}}\left(1-\left(\frac{\mathrm{div}(X_{0})}{2}+\frac{\|X_{0}(x_{0})\|^{2}}{2}-\frac{S(x_{0})}{6}\right)t+O(t^{2})\right),
\end{equation}
where $S$ is the scalar curvature of the Riemannian metric $g$.

Let us now consider the second order H\"{o}rmander-type operator 
\begin{equation}\label{eq:L}
L=\frac12 \sum_{i=1}^{k}X_{i}+X_{0}
\end{equation}
satisfying the \emph{H\"ormander hypoellipticity condition} 
\begin{equation} \label{eq:lie} \tag{wHC}
\mathrm{Lie}\{(\mathrm{ad} X_{0})^{j} X_{i}\,|\, i=1,\ldots,k, \, j\in \N\}\big|_{x}=T_{x}M,\qquad \forall\, x\in M.
\end{equation}
(Here $(\mathrm{ad} X)Y=[X,Y]$ and $\mathrm{Lie}\,\mathcal F$ denotes the smallest Lie algebra containing a family of vector fields $\mathcal{F}$).

Condition \eqref{eq:lie} guarantees the existence of a smooth heat kernel $p(t,x,y)$ for the heat equation associated with $L$ defined in \eqref{eq:L}. Condition \eqref{eq:lie} is also sometimes referred as \emph{weak H\"ormander condition}, in contrast to its \emph{strong} version \eqref{eq:lies} where only the vector field appearing in the second order term of $L$ are required to generate the tangent space 
 \begin{equation} \label{eq:lies} \tag{sHC}
\mathrm{Lie}\{ X_{i}\,|\, i=1,\ldots,k \}\big|_{x}=T_{x}M,\qquad \forall\, x\in M.
\end{equation}
%Let us stress that is only under the assumption \eqref{eq:lies} that it is possible to extract a Carnot-Caratheodory metric to the principal

As soon as one removes the ellipticity assumptions on $L$, the geometric content of the coefficients in the asymptotic expansion of the heat kernel (and indeed even the structure of this asymptotics) is much less understood.

For instance, when the operator $L$ is elliptic, the decay on the diagonal of $p(t,x,x)$ for $t\to 0$ is always of polynomial type, for every drift $X_{0}$ (cf.\ \eqref{eq:2}). On the other hand, when $L$ is only hypoelliptic, the drift plays a crucial role in the structure of the asymptotics.

Due to hypoellipticity, the diffusion generated by the second order term acts in all directions but in an anisotropic way. If the drift field is a smooth section of $\mathcal{G}_2(x):=\mathrm{span}_{x}\{X_{i},[X_{j},X_{k}]: i,j,k\geq 1\}$, then it modifies the coefficients of the diffusion, but only at higher terms and does not affect the order of the asymptotics. On the other hand, if the drift field points outside $\mathcal{G}_2$, it has a stronger effect on the diffusion and its flow could not be compensated by the fields $X_1,\ldots,X_k$, generating a rapid decrease of the heat at the point (see also \cite{article:BAL1}). This different behavior, given by second order Lie brackets, can be heuristically explained by the fact that the second order part of the operator reproduces a Brownian motion, which moves as $\sqrt{t}$, while the drift field has velocity $t$, then if $X_0$ points outside $\mathcal{G}_2$ the diffusion generated by the second order part is too slow.

More precisely, when the drift vector field is either zero or horizontal, i.e., $X_{0}=\sum f_{i}X_{i}$ for some smooth functions $f_{i}$, conditions \eqref{eq:lie} and \eqref{eq:lies} coincide and Ben Arous proved in \cite{benarousdiag} that the decay on the diagonal is polynomial\footnote{Here and in what follows $f(t)\sim g(t)$ means  $f(t)=g(t)(C+o(1))$ for some $C\neq 0$ and $t\to 0$.}
$$p(t,x_{0},x_{0})\sim t^{-Q/2}$$
where $Q$ is the Hausdorff dimension associated to the Carnot-Carath\'eodory metric defined by $X_{1},\ldots,X_{k}$. More generally, in presence of a non-zero drift and under the weak H\"ormander condition assumption \eqref{eq:lie}, in \cite{paoli} it is proved that either
\begin{equation} \label{eq:paoli0}
p(t,x_{0},x_{0})\sim t^{-\mathcal{N}/2}
\end{equation}
or $p(t,x_0,x_0)$ has a faster decay, even exponential for every $x_0$ that is an equilibrium point for the drift. Here the difference in the behaviour is characterized by the controllability of an associated control problem. Moreover $\mathcal{N}$ is an integer that depends on the structure of the Lie algebra \eqref{eq:lie}, that coincides with $Q$ when $X_{0}$ is horizontal. 

%The asymptotics \eqref{eq:paoli0} remains true if  $X_{0}$ is a smooth section of $\mathcal{G}_2$.

When $X_{0}(x_{0})\neq 0$ does not belong to $\mathcal{G}_2(x_{0})$, then the flow of the drift locally interacts with the heat diffusion given by the second order part and generates an exponential decay, see \cite{article:BAL1,article:BAL2}.

%  and this could induce an exponential decay, see \cite{article:BAL1,article:BAL2} and \cite[Example 7]{paoli}. In \cite{article:BAL1,article:BAL2} the authors assume that the horizontal vector fields also satisfy the H\"{o}rmander condition, while in \cite{paoli}

Concerning the geometric meaning of the coefficients of the asymptotic expansion, few results are available and only when the drift field is either zero or horizontal. In \cite{Article:Barilari} it is computed the first term of the asymptotic for 3D contact structures, where an invariant $\kappa$ of the sub-Riemannian structure playing the role of the curvature appears. Concerning higher dimensional structures,  to our best knowledge, the only known results are \cite{stanton84} for the case of a Sasakian manifold (where the trace of the Tanaka Webster curvature appears) and the case of the two higher dimensional model spaces: CR spheres \cite{baudoinwang13} and quaternionic Hopf fibrations \cite{baudoinwang14}. 
 
%If the point $x_{0}$ at which we compute the heat kernel expansion is an equilibrium point for the drift field $X_{0}$, then the decay is still polynomial.  Indeed it depends on how many brackets of the horizontal fields are necessary to generate the direction $X_{0}(x_{0})$ (see \cite{article:BAL1, article:BAL2,paoli}).See  for the case when $X_{0}=0$ 

\medskip
In this paper we perform the first step in the understanding the structure of the asymptotic expansion at the diagonal in the case of an operator with general drift. In particular we are interested in the geometric characterization of its coefficients, for the model class of H\"ormander operators in $\R^{n}$ where $X_{0}$ is a linear field and $X_{i}$ are constant and linearly independent:
$$X_{0}=\sum_{j,h=1}^{n}a_{jh}x_{h}\partial_{x_{j}},\qquad X_{i}=\sum_{j=1}^{n}b_{ij}\partial_{x_{j}},\qquad i=1,\ldots,k.$$
If we denote by $A=(a_{jh})$ and $B=(b_{ij})$ respectively the $n\times n$ and $n\times k$ matrices of the coefficients the operator $L$ assumes the form
 \begin{equation}\label{eq:pdelin0}
L=\sum_{j=1}^n(Ax)_j\frac{\partial}{\partial x_j} + \frac12\sum_{j, h=1}^n (BB^*)_{jh}\frac{\partial^2}{\partial x_j\partial x_h}.% \qquad \mbox{ for } \varphi\in C^\infty_0(\bbR\times\bbR^n)
\end{equation}
The (weak) H\"{o}rmander condition for the operator \eqref{eq:pdelin0} is equivalent to the assumption that there exists a minimal $m\in \N$ such that
\begin{equation}\label{eq:kalman0}
\mathrm{rk}[B,AB,A^2B,\ldots,A^{m-1}B]=n.
\end{equation}
The heat equation associated to $L$ admits a smooth fundamental solution $p(t,x,y)\in C^\infty (\bbR^+\times\bbR^n\times\bbR^n)$ that can be computed explicitly as follows
\begin{equation}\label{eq:kernel}
p(t,x,y)=\frac{e^{-\frac{1}{2}(y-e^{tA}x)^*D_t^{-1}(y-e^{tA}x)}}{(2\pi)^{n/2}\sqrt{\det{D_t}}} %\qquad\qquad \mbox{ for } t>0, \forall x,y\in\bbR^n
\end{equation}
where 
\begin{equation}
D_t=e^{tA}\left(\int_{0}^t e^{-\tau A}BB^* e^{-\tau A^*}d\tau\right) e^{tA^*}.
\label{eq:Dt3}
\end{equation}
By condition \eqref{eq:kalman0}, the matrix $D_t$ is invertible for every $t>0$. 

These operators are the simplest class of hypoelliptic, but not elliptic, operators satisfying \eqref{eq:lie} and are classical in the literature, starting from the pioneering work of H\"{o}rmander \cite{article:Hormander} (see also \cite{LanconelliPolidoro94} for a detailed discussion on this class of operators). Their fundamental solution is the probability density of the stochastic process $\xi_t$ solution of
$$d\xi_t=A\xi_tdt+B dw(t),$$
where $w(t)$ is a $k$-dimensional Brownian motion. For the reader's convenience we present in Appendix the construction of the fundamental solution and its relation to the associated control problem described below.

Indeed, as already pointed out by Stroock and Varadhan \cite{art:stroockvaradhan} in the study of the support of the diffusion, 
the properties of $\xi_t$ are strongly related with the solutions to the control problem% The operator $L$ is intimately related to the following optimal control problem
\begin{equation}\label{eq:le}
\dot{x}(t)=Ax(t)+ Bu(t),\qquad x\in \R^{n}, u\in \R^{k}.
\end{equation}
The relation between the solution of the heat equation and the associated control problem is classical, see for instance \cite{bismut}.
In this paper we further investigate this relation and we show how geometric invariants associated to the \emph{cost functional}%where we minimize the quadratic \emph{cost functional}
\begin{equation}
J_T(u)=\frac{1}{2}\int_0^T |u(s)|^2 ds
\label{eq:cost}
\end{equation}
reveals some geometric-like properties of the heat kernel. In other words, for every fixed $x_{1},x_{2} \in \bbR^n$ and $T>0$, one is interested in computing 
\begin{equation*}
S_T(x_{1},x_{2}):=\inf\{J_T(u): u\in L^\infty([0,T];\bbR^k), x_u(0)=x_{1}, x_u(T)=x_{2}\}.
%\label{eq:optimal cost}
\end{equation*}
where $x_{u}(\cdot)$ is the solution of \eqref{eq:le} associated with the control $u$.
The condition \eqref{eq:kalman0} (also known as \emph{Kalman condition}) ensures that the control system \eqref{eq:le} is controllable, i.e., $S_T(x_{1},x_{2})<+\infty$ for all $x_{1},x_{2}\in \R^{n}$ and $T>0$. 

For $x_0\in \bbR^n$ fixed, let $x_{\bar{u}}(t)$ be an optimal trajectory starting at $x_0$, i.e., a minimizer of the cost functional. The \emph{geodesic cost} associated with $x_{\bar{u}}$ is the family of functions
\begin{equation}
c_t(x):=-S_t(x,x_{\bar{u}}(t)) \qquad \mbox{ for } t>0,x\in\bbR^n.
\label{eq:geodesic cost}
\end{equation}
From the asymptotics of $c_{t}$, one can highlight some ``curvature-like'' invariants of the cost, which define a family of symmetric operators% One can define some ``curvature-like'' invariants of the cost from the function $c_{t}$, defining a family of symmetric operators
\bqn
\mathcal{I}:\R^{k}\to \R^{k},\qquad \mathcal{Q}^{(i)}:\R^{k}\to \R^{k},\qquad i\geq 0.
\eqn
These operators, that are in principle associated with an optimal trajectory, in the case of a linear-quadratic optimal control problem are constant. 
 
The  operator $\mathcal{I}$ is connected to the flag generated by the brackets along the optimal trajectory. The operators $\mathcal{Q}^{(i)}$ play the role of curvature invariants for the optimal control problem (see Section \ref{s:invarianti} and \cite{article:AgrachevBarilariRizzi} for more details). 

We recall here the geometric meaning of the trace of $\mathcal{I}$ that is crucial in what follows. Define the filtration $E_1\subset E_2\subset\ldots \subset E_m=\bbR^n$ as
\begin{equation}\label{eq:E}
E_i=\mathrm{span}\{A^jBx\,|\, x\in \R^{k}, \; 0\leq j\leq i-1\}.
\end{equation}
Then
\begin{equation}\label{eq:NN}
\mathcal{N}:=\mathrm{tr}(\mathcal{I})
= \sum_{i=0}^{m-1} (2i-1)(\dim E_{i+1}-\dim E_{i}). 
\end{equation}

\subsection{Main results}
When $x_{0}$ is an equilibrium of the drift field, we can compute and characterize all the coefficients in the small time asymptotic expansion, providing a characterization of the coefficients that is analogous to the one obtained on a Riemannian manifold. %From the results of \cite{paoli?} we know that

%Notice that the coefficients do not depend on the point $x_{0}$.
\begin{theorem}\label{th:Ker} Assume that $Ax_{0}=0$. Then
\begin{equation}\label{eq33}
p(t,x_{0},x_{0})=\frac{ t^{-\mathcal{N}/2}}{(2\pi)^{n/2}\sqrt{c_0}}\left(\sum_{i=0}^{m}a_{i} t^{i}+O(t^{m})\right), \qquad \mbox{ for } t\to 0,
\end{equation}
where $\mathcal{N}=\mathrm{tr}(\mathcal{I})$ is defined in \eqref{eq:NN} and $c_0$ is a positive constant.
Moreover there exists universal polynomials $P_{i}$ of degree $i$ such that 
$$a_{i}=P_{i}(\tr A, \tr \calQ^{(0)},\ldots, \tr \calQ^{(i-2)}).$$
In particular for $i=1,2,3,$ we have
\begin{gather}
a_{1}=-\frac{\tr A}{2},\qquad 
a_{2}=\frac{(\tr A)^{2}}{8}+\frac{\tr \calQ^{(0)}}{4},\\[0,2cm]
a_{3}=-\frac{\tr \calQ^{(1)}}{12}-\frac{\tr A \,\tr \calQ^{(0)}}{8} -\frac{(\tr A)^{3}}{48} .
\end{gather}
%Ho cambiato i coefficienti, mi sembra che siano giusti cos\`{i}. Verificare il precedente teorema!
\end{theorem}
We stress that the explicit structure of any higher order coefficient can be a priori computed by a simple Taylor expansion, as it follows from the proof, cf.\ Section \ref{s:proof1}.

More in general, one has an expansion of $p(t,x,y)$, at every pair of points $x,y$, relating the heat kernel with the optimal cost functional and the same geometric coefficients of Theorem \ref{th:Ker}.
\begin{corollary}\label{th:cost}
For any pair of points, $x,y\in\bbR^n$,
$$p(t,x,y)=\frac{ t^{-\mathcal{N}/2}}{(2\pi)^{n/2}\sqrt{c_0} }e^{-S_t(x,y)}\left(\sum_{i=0}^{m}a_{i} t^{i}+O(t^{m})\right), \qquad \mbox{ for } t\to 0,$$
where the coefficients $a_{i}$ are characterized as in Theorem \ref{th:Ker}.
\end{corollary}
This corollary is a direct consequence of the previous theorem and is proved in Section \ref{s:proof1b}.

\medskip
Next we consider the case when $x_{0}$ is not a zero of the drift field. In this case, as explained, one can observe different behaviors depending on the smallest level of the filtration \eqref{eq:E} to which the vector $Ax_{0}$ belongs. Indeed the optimal cost to remain fixed in the point $x_0$ is strictly positive and the asymptotics depends on the exponential in Theorem \ref{th:cost}.
%Therefore $E_i$ has dimension $k_i$.
\begin{theorem}\label{th:nonKer} Assume that $Ax_{0}\neq 0$. Then
\bi
\iii[(i)] If $Ax_{0}\in E_{1}$, then we have the
 polynomial decay 
$$p(t,x_{0},x_{0})=\frac{ t^{-\mathcal{N}/2}}{(2\pi)^{n/2}\sqrt{c_0}}\left[1-\left(\frac{\tr A}{2}+\frac{|Ax_{0}|^{2}}{2}\right)t+O(t^2)\right], \quad \mbox{ for } t\to 0.$$
\iii[(ii)] If $Ax_0\in E_i\setminus E_{i-1}$ for some $i>1$, then $p(t,x_{0},x_{0})$ has exponential decay to zero. More precisely there exists $C>0$ such that
$$p(t,x_0,x_0)= \frac{ t^{-\mathcal{N}/2}}{(2\pi)^{n/2}\sqrt{c_0}}\exp \left(-\frac{C+O(t)}{t^{2i-3}}\right), \qquad \mbox{ for } t\to 0.$$
\ei
\end{theorem}
We stress that claim (i) is the analogue of the expansion \eqref{eq:2}: here no scalar curvature appears since all brackets of horizontal fields are zero. For the same reason, in this case the space $\mathcal{G}_{2}$ coincides with $E_{1}$ and in claim (ii), when $Ax_{0}\notin E_{1}$, we indeed observe an exponential decay.
The proof of Theorem \ref{th:nonKer} is presented in Section \ref{s:proof2}.

%\subsection{Structure of the paper} 

%\subsection{extra/ da posticipare}
%\begin{equation}
%\frac{\partial\varphi}{\partial t}-\sum_{j=1}^n(Ax)_j\frac{\partial\varphi}{\partial x_j} -\frac{1}{2} \sum_{j, h=1}^n (BB^*)_{jh}\frac{\partial^2\varphi}{\partial x_j\partial x_h} \qquad \forall \varphi\in C^\infty_0(\bbR^n).
%\label{eq:pde}
%\end{equation}
%is hypoelliptic and admits a well known This is given by the distribution of the solution, $\xi_t$, of the associated stochastic differential equation in the Stratonovich form starting at a point $x\in\bbR^n$:
%$$d\xi_t=A\xi_t dt+B\circ d w_t,\qquad \xi_0=x,$$
%where $w_t=(w^1_t,\ldots,w^k_t)$ is a $k$- dimensional Brownian motion. 
%The solution of this linear stochastic differential equation is well-known and is the Gaussian distribution on $\bbR^n$
%\begin{equation}
%p(t,x,y)=P(\xi_t=y|\xi_0=x)=\frac{e^{-\frac{1}{2}(y-e^{tA}x)^*D_t^{-1}(t-e^{tA}x)}}{(2\pi)^{n/2}\sqrt{\det{D_t}}} \qquad\qquad \forall t>0, \forall x,y\in\bbR^n
%\label{eq:kernel}
%\end{equation}
%whit mean $\mu_{t,x}=e^{tA}x$ and covariant matrix $D_t$ given by
%$$D_t=e^{tA}\int_{0}^t e^{-\tau A}BB^* e^{-\tau A^*}d\tau e^{tA^*}.$$
%

%%%%%%%%%%%%%%%%%%%%%%%%%%%%%%%%%%%%%%%%%%%%%%%%%%%%%%%%%%%%%%%%%%%%%%%%%%%%%%%%%%%%%%%%%%%%%%%%%%%%%%%%%%%%%%%%%%%%%%%%%%%%%%%%%%%%%%%%%%%%%
\section{Linear quadratic optimal control problems}
%Our aim is to give a geometric characterization of the coefficients of the small time asymptotics of the heat kernel on the diagonal. This interpretation will arise from the asymptotics of the cost of an associated optimal control problem.
Let us consider the optimal control problem associated with the operator $L$
\begin{equation}
\left\{
\begin{array}{l}
\dot{x}=Ax+ Bu\\
J_T(u)=\frac{1}{2}\int_0^T \sum_{i=1}^k|u_i(s)|^2 ds\to \min
\end{array}
\right.
\label{eq:ocp}
\end{equation}
Here $u\in L^\infty([0,T];\bbR^k)$ is the control and $J_T$ is the optimal cost to be minimized. A curve $x(t)$ in $\bbR^n$ is called \emph{admissible} for the control problem \eqref{eq:ocp} if there exists a control function $u\in L^\infty([0,T];\bbR^k)$ such that $\dot{x}(t)=Ax(t)+ Bu(t)$ for a.e. $t\in[0,T]$.

The solution of the differential equation \eqref{eq:ocp} corresponding to the control $u$ will be denoted by $x_u:[0,T]\rightarrow\bbR^n$ and for a fixed initial point $x_1\in \bbR^n$ is given by:
\begin{equation}
x_u(t)=e^{tA}x_1+e^{tA}\int_0^t e^{-\tau A}Bu(\tau)d\tau.
\label{eq:Cauchy}
\end{equation}
Among all trajectories $x_u$ starting at $x_1$ and arriving in a point $x_2\in\bbR^n$ in time $T$ we want to minimize the cost functional $J_T$:
for every fixed $x_1,x_2\in \bbR^n$ and $T>0$, we define the \emph{value function}
\begin{equation}
S_T(x_1,x_2):=\inf\{J_T(u)| u\in L^\infty([0,T];\bbR^k), x_u(0)=x_1, x_u(T)=x_2\}.
\label{eq:optimal cost2}
\end{equation}
A control $\bar{u}$ that realizes the minimum in \eqref{eq:optimal cost2} is called an \emph{optimal control}, and the corresponding trajectory $x_{\bar{u}}:[0,T]\rightarrow\bbR^n$ is called an \emph{optimal trajectory} of the control problem \eqref{eq:ocp}.

It is well-known (see for example \cite{book:Agrachev}) that the optimal trajectories of the control problem \eqref{eq:ocp} can be obtained as the projection of the solutions of an Hamiltonian system in $T^*\bbR^n$. Namely, let
$$H(p,x)=p^*Ax+\frac{1}{2} p^*BB^*p, \qquad \forall\, (p,x)\in T^*\bbR^n,$$
be the Hamiltonian function associated with the optimal control problem. All the optimal trajectories are the projection $x(t)$ of the solution $(p(t),x(t))\in T^*\bbR^n\cong \bbR^{2n}$ of the Hamiltonian system associated with $H$
\begin{equation}
\begin{cases}
\dot{p}=-A^*p\\
\dot{x}=Ax+BB^*p.
\end{cases}
\label{eq:hsystem}
\end{equation}
Moreover, the control realizing the optimal trajectory is uniquely recovered by $\bar u(t)=B^*p(t)$. Thus the solution corresponding to the initial condition $(p_0,x_0)\in T^*_{x_0}\bbR^n$ can be found explicitly
\begin{equation} 
\begin{cases}
p(t)=e^{-t A^*}p_0\\
x(t)=e^{tA}\left(x_0+\int_0^te^{-\tau A}BB^* e^{-\tau A^*}d\tau p_0\right).
\end{cases}
\label{eq:extremals}
\end{equation}
Let us denote by $\Gamma_t$ the matrix
\begin{equation}
\Gamma_t:=\int_0^te^{-\tau A}BB^* e^{-\tau A^*}d\tau.
\label{eq:Gamma}
\end{equation}
By Kalman's condition \eqref{eq:kalman0}, it follows that $\Gamma_t$ is invertible for every $t>0$.

\begin{remark}\label{rk:cost function}
Fix $x_1,x_2\in\bbR^n$ and $T>0$. By the explicit formulas \eqref{eq:extremals} there exists a unique initial covector $p_0$ such that the corresponding extremal $x(t)$ satisfies $x(0)=x_1$ and $x(T)=x_2$. It is equal to
$$p_0=\Gamma_T^{-1}\left(e^{-TA}x_2-x_1\right).$$
Since the optimal control is given by $\bar u(t)=B^*p(t)$, we can also write the optimal cost to go from $x_1$ to $x_2$, namely
$$S_T(x_1,x_2)=\frac{1}{2}p_0^*\Gamma_T p_0.$$
It follows that the cost function is smooth in $(T,x_1,x_2)\in\bbR^+\times\bbR^n\times\bbR^n$, where $\bbR^+=\{T\in \R:T>0\}$.
\end{remark}

%%%%%%%%%%%%%%%%%%%%%%%%%%%%%%%%%%%%%%%%%%%%%%%%%%%%%%%%%%%%%%%%%%%%%%%%%%%%%%%%%%%%%%%%%%%%%%%%%%%%%%%%%%%%%%%%%%%%%%%%%%%%%%%%%%%
\section{The flag and growth vector of an admissible curve}
Let $x_u:[0,T]\rightarrow\bbR^n$ be an admissible curve such that $x_u(0)=x_0$, associated with the control $u$. Let $P_{0,t}$ be the flow defined by $u$, i.e., for every $y\in\bbR^n$
$$P_{0,t}(y):=x_u(t;y)\quad \mbox{ s.t. }\quad x_u(0;y)=y.$$

At any point of $\bbR^n$ we split the tangent space $T_x\bbR^n\cong\bbR^n=D\oplus D^{\perp}$, where $D$ is the $k$-dimensional susbspace generated by the columns of $B$ and $D^{\perp}$ is its orthogonal complement, and we define the following family of subspaces of $T_{x_0}\bbR^n$:
\begin{equation}
\calF_{x_u}(t):=(P_{0,t})_{*}^{-1}D \subset T_{x_0}\bbR^n .
\label{eq:F}
\end{equation} 
In other words, the family $\calF_{x_u}(t)$ is obtained by translating in $T_{x_0}\bbR^n$ the subspace $D$ along the trajectory $x_u$ using the flow $P_{0,t}$.
\begin{definition}
The \emph{flag of the admissible curve} $x_u(t)$ is the sequence of subspaces
\begin{equation}
\calF_{x_u}^i(t):=\mathrm{span}\left\{\left.\frac{d^j}{dt^j}v(t)\right| v(t)\in \calF_{x_u}(t) \mbox{ smooth, } j\leq i-1\right\}\subset T_{x_0}\bbR^n,\quad i\geq 1.
\label{eq:flag}
\end{equation}
\end{definition}
By construction, this is a filtration of $T_{x_0}\bbR^n$, i.e., $\calF_{x_u}^i(t)\subset \calF_{x_u}^{i+1}(t)$, for all $i\geq 1$.

\begin{definition}
Denote by $k_i(t):=\dim \calF_{x_u}^i(t)$. The \emph{growth vector of the admissible curve $x_u(t)$} is the sequence of integers
$$\calG_{x_u}(t)=\left\{k_1(t),k_2(t),\ldots\right\}.$$
\end{definition}
\begin{remark}\label{rk:shift}
For any $s\in[0,T]$ we can define the family of subspace $\calF_{x_{u,s}}(t)$ associated with the admissible curve $x_u$ starting at time $s$. Namely, let $P_{s,t}$ be the flow defined by $u$ starting at time $s<t$, i.e., for every $y\in\bbR^n$
$$P_{s,t}(y):=x_u(t;y)\qquad \mbox{ s.t. } x_u(s;y)=y,$$
and let $x_{u,s}(t):=x_u(s+t)$ be the shifted curve by time $s$. Then we introduce the family of subspaces $\calF_{x_{u,s}}(t):=(P_{s,s+t})_*^{-1}\bbR^k$ with base point $x_u(s)$.

The relation $\calF_{x_{u,s}}(t)=(P_{0,s})_*\calF_{x_u}(s+t)$ implies that the growth vector of the original curve at time $t$ can be equivalently computed via the growth vector at time $0$ of the shifted curve $x_{u,t}$, i.e., $k_i(t)=\dim \calF_{x_{u,t}}^i(0)$ and $\calG_{x_u}(t)=\calG_{x_{u,t}}(0)$.
\end{remark}

The growth vector of a curve $x_u$ at time $0$ can be easily computed. Indeed, by the explicit expression \eqref{eq:Cauchy} of the flow of $u$, the flag of the curve $x_u$ starting from a point $x_0$ is
$$\calF^i_{x_u}(0)=\mathrm{span}\{B,AB,\ldots,A^{i-1}B\}.$$
In particular the flag at time $0$ is independent on the control $u$ and the initial point $x_0$. By Remark \ref{rk:shift} the growth vector of any curve $x_u$ is then independent also from the time and is equal to $\calG_{x_u}(t)=\left\{k_1,k_2,\ldots\right\}$, where the indices $k_i$ are
$$k_i:=\dim \mathrm{span}\{B,AB,\ldots,A^{i-1}B\}.$$
By Kalman's condition \eqref{eq:kalman0} there exists a minimal integer $1\leq m\leq n$ such that $k_m=n$. We call such $m$ the \emph{step} of the control problem (independent of the admissible curves). Moreover notice that $k_1=k$.

\begin{lemma}
Let $d_i:=k_i-k_{i-1}$. Then $d_1\geq d_2\geq \ldots\geq d_m$.
\end{lemma}
\begin{proof}
The linear map $\widehat{A}:\calF^i_{x_u}(t)\rightarrow \calF^{i+1}_{x_u}(t)/\calF^i_{x_u}(t)$ defined by $\widehat{A}v:=Av$ for every $v\in \calF^i_{x_u}(t)$ is surjective and $\mathrm{Ker}\widehat{A}=\calF^{i-1}_{x_u}(t)$. Then
$$\dim \calF^i_{x_u}(t)-\dim\calF^{i-1}_{x_u}(t)\geq \dim \calF^{i+1}_{x_u}(t)-\dim \calF^i_{x_u}(t),$$
which concludes the proof.
\end{proof}

To any curve $x_u$ we associate a tableau with $m$ columns of length $d_i$ for $i=1,\ldots,m$. By the previous Lemma, the height of the columns is decreasing from left to right. We call $n_j$ the length of the $j-$th row, for $j=1,\ldots, k$ (for example $n_1=m$, see Figure \ref{table:tableau}).
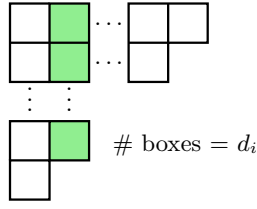
\begin{figure} 
\begin{center}
\begin{tikzpicture}[x=0.26mm, y=0.26mm, inner xsep=0pt, inner ysep=0pt, outer xsep=0pt, outer ysep=0pt]
\path[line width=0mm] (61.05,70.00) rectangle +(127.98,100.00);
\draw(120.00,159.00) node[anchor=base]{\fontsize{9.39}{11.27}\selectfont $\ldots$};
\draw(120.00,139.00) node[anchor=base]{\fontsize{9.39}{11.27}\selectfont $\ldots$};
\draw(80.00,115.00) node[anchor=base]{\fontsize{9.39}{11.27}\selectfont $\vdots$};
\draw(100.00,115.00) node[anchor=base]{\fontsize{9.39}{11.27}\selectfont $\vdots$};
\definecolor{L}{rgb}{0,0,0}
\path[line width=0.30mm, draw=L] (130.00,170.00) -- (130.00,130.00) -- (150.00,130.00) -- (150.00,150.00) -- (170.00,150.00) -- (170.00,170.00) -- cycle;
\path[line width=0.30mm, draw=L] (150.00,170.00) -- (150.00,150.00);
\path[line width=0.30mm, draw=L] (130.00,150.00) -- (150.00,150.00);
\draw(122.00,95.00) node[anchor=base west]{\fontsize{9}{10.24}\selectfont \# boxes = $d_i$};
\definecolor{F}{rgb}{0.565,0.933,0.565}
\path[line width=0.30mm, draw=L, fill=F] (90.00,170.00) [rotate around={270:(90.00,170.00)}] rectangle +(40.00,20.00);
\path[line width=0.30mm, draw=L] (90.00,170.00) -- (70.00,170.00) -- (70.00,130.00) -- (90.00,130.00);
\path[line width=0.30mm, draw=L] (70.00,150.00) -- (110.00,150.00);
\path[line width=0.30mm, draw=L, fill=F] (90.00,110.00) [rotate around={270:(90.00,110.00)}] rectangle +(20.00,20.00);
\path[line width=0.30mm, draw=L] (90.00,70.00) [rotate around={90:(90.00,70.00)}] rectangle +(20.00,20.00);
\path[line width=0.30mm, draw=L] (70.00,90.00) -- (70.00,110.00) -- (90.00,110.00);
\end{tikzpicture}%
\end{center}
\caption{Young diagram}\label{table:tableau}
\end{figure}
%\begin{table}[h]
%\centering
%\begin{tabular}{m{0.5cm}m{0.5cm}m{0.5cm}m{0.5cm}m{0.5cm}m{0.5cm}}
%\cline{2-3} \cline{5-6}
%\multicolumn{1}{l|}{$n_1$}     & \multicolumn{1}{l|}{\textcolor{white}{${a\choose b}$}} & \multicolumn{1}{l|}{} & \multicolumn{1}{l|}{$\cdots$} & \multicolumn{1}{l|}{} & \multicolumn{1}{l|}{} \\ [0.5cm] \cline{2-3} \cline{5-6} 
%\multicolumn{1}{l|}{$n_2$}     & \multicolumn{1}{l|}{} & \multicolumn{1}{l|}{} & \multicolumn{1}{l|}{$\cdots$} & \multicolumn{1}{l|}{} &  $d_m$                \\[0.5cm] \cline{2-3} \cline{5-5}
%                               &  $\vdots$             &         $\vdots$      &                               &    $d_{m-1}$          &                       \\[0.5cm] \cline{2-3}
%\multicolumn{1}{l|}{$n_{k-1}$} & \multicolumn{1}{l|}{} & \multicolumn{1}{l|}{} &                               &                       &                       \\[0.5cm] \cline{2-3}
%\multicolumn{1}{l|}{$n_k$}     & \multicolumn{1}{l|}{} &  $d_2$                &                               &                       &                       \\[0.5cm] \cline{2-2}
%                               &      $d_1$            &                       &                               &                       &                      
%\end{tabular}
%\caption{Tableau of an admissible curve $x_u$.}
%\label{table:tableau}
%\end{table}

%%%%%%%%%%%%%%%%%%%%%%%%%%%%%%%%%%%%%%%%%%%%%%%%%%%%%%%%%%%%%%%%%%%%%%%%%%%%%%%%%%%%%%%%%%%%%%%%%%%%%%%%%%%%%%%%%%%%%%%%%%%%%%%%%%%%%%%%%%%%
\section{Geodesic cost and curvature invariants} \label{s:invarianti}
\begin{definition}
Let $x_0\in \bbR^n$ and fix an optimal trajectory $x_{\bar{u}}(t)$ starting at $x_0$ of the optimal control problem \eqref{eq:ocp}. The \emph{geodesic cost} associated with $x_{\bar{u}}$ is the family of functions $\{c_{t}\}_{t>0}$ defined by
\begin{equation}
c_t(x):=-S_t(x,x_{\bar{u}}(t)), \qquad x\in\bbR^n,
\label{eq:geodesic cost2}
\end{equation}
where $S_t$ is the optimal cost function defined in \eqref{eq:optimal cost2}.
\end{definition}

Notice that thanks to Remark \ref{rk:cost function} and the smoothness of optimal trajectories, the geodesic cost is smoooth in $\bbR^+\times\bbR^n$. 

Moreover, for any $t>0$ and $x\in\bbR^n$ there exists a unique minimizer of the cost functional among all the trajectories that connect $x$ with $x_{\bar{u}}(t)$ in time $t$. By the explicit expressions of the extremals in \eqref{eq:extremals} and of the optimal control $\bar u$, we can write the explicit formula:
$$c_t(x)=-\frac{1}{2}p_0^*\Gamma_t p_0+p_0^*(x-x_0)-\frac{1}{2}(x-x_0)^*\Gamma_t^{-1}(x-x_0),$$
where $x_{\bar u}$ is a solution of the Hamiltonian system with initial data $(p_{0},x_{0})$ and  $\Gamma_t$ is the invertible matrix defined in \eqref{eq:Gamma}.

%We recall here, for the linear quadratic optimal control problem \eqref{eq:le}, the definition of \emph{curvature} of a geodesic, that was introduced in \cite{article:AgrachevBarilariRizzi} Chapter 4 for a general class of affine optimal control problems.

%Let $x_0\in\bbR^n$ and fix a geodesic $x(t)=\pi e^{t\vec{H}}(p_0,x_0)\in\bbR^n$ starting at $x_0$ with initial covector $\lambda=(p_0,x_0)$. Let us split the space $\bbR^n=\bbR^k\times \bbR^{n-k}$, where $k$ is the rank of the matrix $B$, and define the following family of quadratic forms, $\calQ_\lambda(t)$, on $\bbR^k$:
%Let $(p_0,x_0)$ be the initial covector associated with $x_{\bar{u}}$. 
Then we define the following family of quadratic forms, $\calQ (t)$, on $\bbR^k$:
\begin{equation}
\calQ (t):=B^*\left(d^2_{x_0}\dot{c}_t\right)B=-\frac{d}{dt}B^*\Gamma_t^{-1}B.
\label{eq:Qlambda}
\end{equation}
This family of operators is the linear quadratic counterpart of the more general family of operators introduced in  \cite[Chapter 4]{article:AgrachevBarilariRizzi} for the wider class of non linear control systems that are affine in the control.
\begin{remark} Actually the family of operators $\mathcal{Q}(t)$ does not depend on the initial data $(p_{0},x_{0})$ of the optimal trajectory and is the same for any geodesic. This is saying that it is an intrinsic object of the pair  control system and cost.

Moreover, let us stress that $\calQ(t)$ is an intrinsic object of the optimal control problem, i.e., it does not depend on the chosen coordinate on $\bbR^n$.

Indeed let $y=Cx$, with $C$ an $n\times n$ invertible matrix, be a linear change of coordinates on $\bbR^n$. In the new coordinates the dynamical system \eqref{eq:ocp} is rewritten as
$$\dot{y}=\tilde{A}y+\tilde{B}u,$$
where $\tilde{A}=CAC^{-1}$ and $\tilde{B}=CB$. Since the matrix $C$ is invertible the dimensions $d_i$ remain unchanged, and in particular $\rk \, B=\rk\,  \tilde{B}=k$. The matrix $\tilde{\Gamma}_t$ for the new coordinate system can be easily recovered by definition \eqref{eq:Gamma} and we find $\tilde{\Gamma}_t=C\Gamma_t C^*$. This implies in particular that the definition of $\calQ(t)$ in \eqref{eq:Qlambda} is independent on the coordinates. See also \cite{article:AgrachevBarilariRizzi} for a  general discussion about the geodesic cost.
\end{remark}
%\textcolor{blue}{Mi piacerebbe scrivere la famiglia $\calQ$ in modo intrinseco, senza dover fare una dimostrazione per il cambio di coordinate. Riguardare come ridefinire in modo intrinseco tutti gli elementi.}

The next theorem shows the asymptotic behavior of the family $\calQ(t)$ for small $t>0$. The proof of this result, in its general setting, can be found in \cite[Chapter 4]{article:AgrachevBarilariRizzi}. 

\begin{theorem}[Theorem A in \cite{article:AgrachevBarilariRizzi}]\label{th:A}
Let $x_{\bar{u}}:[0,T]\rightarrow\bbR^n$ be an optimal trajectory. % with initial covector $\lambda\in T^*_{x_0}\bbR^n$, i.e. $x_{\bar{u}}=\pi\circ e^{t\vec{H}}\lambda$. 
The function $t\mapsto t^2\calQ(t)$ defined in Eq.\ \eqref{eq:Qlambda} can be extended to a smooth family of symmetric operators on $\bbR^n$ for all $t\geq 0$. %, such that $\left.\frac{d}{dt}\right|_{t=0}t^2\calQ(t)=0$. 
In particular, for every fixed $h\in\bbN$, one has the following Laurent expansion
\begin{equation}
\calQ(t)=\frac{1}{t^2}\calI+\sum_{i=0}^h\calQ^{(i)}t^i+O\left(t^{h+1}\right),\qquad t>0.
\label{eq:Q}
\end{equation}
Moreover all the matrices $\calI$ and $\calQ^{(i)}$ are symmetric and 
$$\mathrm{tr}\,\calI=\sum_{i=1}^k n_i^2=\sum_{j=1}^{m}(2j-1)d_{j}.$$
\end{theorem}

The expansion \eqref{eq:Q} defines a sequence of symmetric operators (or matrices) $\calI$ and $\calQ^{(i)}$, for $i\in\bbN$. These operators are canonically associated to the optimal control problem.

%\begin{definition}
%%Let $\lambda\in T^*_{x_0}\bbR^n$ be a the initial covector associated with the extremal trajectory $x_{\bar{u}}$. The \emph{curvature} at the point $\lambda\in T^*_{x_0}\bbR^n$ is the symmetric operator $\calR_\lambda:\bbR^k\rightarrow\bbR^k$ defined as
%We will call \emph{curvature} of the geodesics $x=x_{\bar{u}}$ the symmetric operator $\calR:\bbR^k\rightarrow\bbR^k$ defined as
%\begin{equation}
%\calR:=\frac{3}{2}\left.\frac{d^2}{dt^2}\right|_{t=0} t^2\calQ(t)=3\calQ^{(0)}.
%\label{eq:curvature}
%\end{equation}
%The \emph{Ricci curvature} is defined as $\mathrm{Ric}:=\tr \,\calR$.
%\end{definition}

%%%%%%%%%%%%%%%%%%%%%%%%%%%%%%%%%%%%%%%%%%%%%%%%%%%%%%%%%%%%%%%%%%%%%%%%%%%%%%%%%%%%%%%%%%%%%%%%%%%%%%%%%%%%%%%%%%%%%%%%%
\section{Small time asymptotics at an equilibrium point} \label{s:proof1}
In this section we prove Theorem \ref{th:Ker}, concerning the small time asymptotics of the fundamental solution $p(t,x,x)$ at a point $x\in \mathrm{Ker}A$. We will write this expansion in terms only of the drift field and the invariant $\calQ^{(i)}$ introduced in Theorem \ref{th:A}.
%We will show   is determined only by the drift field and the asymtptotic of the cost function. In particular the first terms of the asymptotic are
%$$p(t,x,x)=\frac{1}{\sqrt{c_0 t^N}}\left[1-\frac{\tr A}{2} t+\left(\frac{(\tr A)^2}{8}+\frac{\tr\calR}{12}\right)t^2+O(t^3)\right].$$
%This result agrees with the asymptotic already found in the simplest cases.
%$$p(t,x,x)=\frac{e^{-\frac{1}{2}\left(\tr A t-\frac{\tr\calR}{6}t^2+o(t^2)\right)}}{\sqrt{c_0 t^N}}$$
%where the terms in $o(t^2)$ depend only on the asymptotic of the cost function associated with the ample fixed curve $\gamma\equiv x$.
%
%\begin{remark}
%The asymptotic is determined only by the drift field and the asymtptotic of the cost function of $\gamma$.
%
%The first terms of the asymptotic are
%$$p(t,x,x)=\frac{1}{\sqrt{c_0 t^N}}\left[1-\frac{\tr A}{2} t+\left(\frac{(\tr A)^2}{8}+\frac{\tr\calR}{12}\right)t^2+o(t^2)\right].$$
%This result agrees with the asymptotic already found in the simplest cases.
%\end{remark}
\begin{remark}
The order of the asymptotics, $\calN$, was already found in \cite{paoli} for a wider class of hypoelliptic operators. In particular, for the linear case the number $\calN$ is determined only by the numbers $k_i=\rk\{B,AB,\ldots,A^{i-1}B\}$, for every $1\leq i\leq m-1$, and $k_0=0$. Indeed $\calN$ can be rewritten as follows
$$\calN:=\sum_{i=1}^m (2i-1)(k_i-k_{i-1}).$$

This number is actually equal to the trace of $\calI$, given in Theorem \ref{th:A}, since by the classical identity $\sum_{i=1}^n (2i-1)=n^2$ we have
$$\mathrm{tr}\,\calI=\sum_{i=1}^k n_i^2=\sum_{i=1}^k\sum_{j=1}^{n_i}(2j-1)=\sum_{j=1}^m(2j-1)(k_j-k_{j-1})=\calN.$$
\end{remark}
%Let
%$$\Gamma_t:=\int_{0}^t e^{-\tau A}BB^* e^{-\tau A^*}d\tau$$
%then, by a change of variable in the integral, it holds
%\begin{equation}
%D_t=e^{tA}\Gamma_t e^{tA^*}=-\Gamma_{-t}.
%\label{eq:2}
%\end{equation}
Recall now the expression of the fundamental solution $p(t,x,y)$ given in \eqref{eq:kernel}:
$$p(t,x,y)=\frac{e^{-\frac{1}{2}(y-e^{tA}x)^*D_t^{-1}(y-e^{tA}x)}}{(2\pi)^{n/2}\sqrt{\det{D_t}}},$$
where the matrix $D_t$  is defined as
\begin{equation}
D_t=e^{tA}\int_{0}^t e^{-\tau A}BB^* e^{-\tau A^*}d\tau e^{tA^*}=e^{tA}\Gamma_t e^{tA^*}.
\label{eq:Dt}
\end{equation}
A change of variable in the integral defining $\Gamma_t$ gives easily that $D_t=-\Gamma_{-t}$.

If a point $x$ is an equilibrium point for the drift field (i.e., $x\in \mathrm{Ker}A$), then the asymptotics of $p(t,x,x)$ on the diagonal is determined uniquely by the Taylor expansion of $\sqrt{\det{D_t}}^{-1}$, since $e^{tA}x=x$ for every $t$. 

Let $d(t):=\det{D_t}$, then $d(t)$ satisfies
\begin{equation}
%\begin{split}
d'(t)=d(t) \mathrm{tr}(D_t^{-1}\dot{D}_t)=%d(t) \mathrm{tr}(D_t^{-1}AD_t+A^*+D_t^{-1}BB^*)\\
d(t)(2 \mathrm{tr}(A)+\mathrm{tr}(B^*D_t^{-1}B)).
\label{eq:derivatadt}
%\end{split}
\end{equation}
Moreover, since $d(t)$ has a simple pole at $t=0$ of order $\calN$ we can write the determinant as $d(t)=t^{\calN} f(t)$, for some smooth  function $f$ non-vanishing at zero. Substituting  this expression in \eqref{eq:derivatadt} one gets
\begin{equation}
%\begin{split}
\frac{d'(t)}{d(t)}=\frac{\calN}{t}+\frac{f'(t)}{f(t)}=2 \mathrm{tr}(A)+\mathrm{tr}(B^*D_t^{-1}B).%\\
%&=2 \mathrm{tr}(A)+\frac{\calN}{t}+\mathrm{tr }(C)+\sum_{i=0}^h(-1)^{i+1}\mathrm{tr}(\calQ^{(i)})\frac{t^{i+1}}{i+1}+O(t^{h+2})
%\end{split}
\label{eq:3b}
\end{equation}
Theorem \ref{th:A} illustrates the asymptotic expansion of $-\frac{d}{dt}\mathrm{tr}(B^*\Gamma_t^{-1}B)$ in terms of the invariants $\mathcal{I}$ and $\calQ^{(i)}$. Its integral is
$$-\mathrm{tr}(B^*\Gamma_t^{-1}B)=-\frac{\calN}{t}+c + \sum_{i=0}^h\mathrm{tr}(\calQ^{(i)})\frac{t^{i+1}}{i+1}+O\left(t^{h+2}\right) $$
for some constant $c$, coming from the integration. Recall that $D_t=-\Gamma_{-t}$. Thus in Eq. \eqref{eq:3b} we have
$$\frac{f'(t)}{f(t)}=2 \mathrm{tr}(A)+c+\sum_{i=0}^h(-1)^{i+1}\mathrm{tr}(\calQ^{(i)})\frac{t^{i+1}}{i+1}+O(t^{h+2}).$$

%Let $S(t):=-\Gamma(t)$ and $S^{\flat^{-1}}_{t}:=B^*S_{t}^{-1}B$. By the asymptotic determining the curvature of an ample equiregular curve, we know that
%$$S^{\flat^{-1}}_{t}=-\frac{\mathcal{I}}{t}+C+\frac{\mathcal{R}}{3}t+o(t),$$
%for some indetermined constant matrix $C$. By equation \eqref{eq:2} $B^*D_t^{-1}B=S^{\flat^{-1}}_{-t}$ and its trace is
%\begin{equation}
%\mathrm{tr}(B^*D_t^{-1}B)=\frac{N}{t}+\mathrm{tr}C-\frac{\mathrm{tr}\mathcal{R}}{3}t+o(t).
%\label{eq:7}
%\end{equation}

Then we can write the determinant of $D_t$ in the following exponential form depending on the invariants $\calQ^{(i)}$
\begin{equation}
\det D_t=c_0 t^\calN e^{(c+2\mathrm{tr}A)t+\sum_{i=0}^h (-1)^{i+1}\tr\mathcal{\calQ}^{(i)}\frac{t^{i+2}}{(i+1)(i+2)}+O(t^{h+3})},
\label{eq:5}
\end{equation}
for some constant $c_0$ and the constant $c:=\mathrm{tr}(C)$.
In particular, one can easily find the first terms of the Taylor expansion of $\det D_t$ at $t=0$, namely
\begin{equation}
\det D_t=c_0 t^\calN\left[1+(c+2\mathrm{tr}A)t+\frac{(c+2\mathrm{tr}A)^2}{2}t^2-\frac{\tr\mathcal{Q}^{(0)}}{2}t^2+o(t^2)\right],
\label{eq:4}
\end{equation}
determined up to a constant $c$.
\begin{remark}
Notice that, in order to have a non degenerate matrix $D_t$, it is essential  that the order $\calN$ of $\det D_t$ is the same term that appears in the asymptotics of $\mathrm{tr}\calQ(t)$.
\end{remark}
\subsection{The first term in the expansion}
To find the constant $c$ in Eq.\ \eqref{eq:4} let us compute the first terms of the expansion of $\det D_t$. 
%The only part that gives some problems is the integral matrix $\Gamma_t$. So let us compute the determinant by writing $\Gamma_t$ in a Taylor series.
%By a change of variables, we can assume that the matrix $B$ is the identity matrix in the first $k$ rows and zero in the last $n-k$ rows.
The derivative matrix $\dot{\Gamma}_t=e^{-tA}BB^*e^{-tA^*}$ is positive semi-definite and can be written as
$$\dot{\Gamma}_t=V(t)V(t)^*,$$
for $V(t)=e^{-tA}B$. Let $v_i(t)$ denote the columns of $V(t)$ and define the filtration $E_1\subset E_2\subset\ldots \subset E_m=\bbR^n$ as
\begin{equation}\label{eq:E2}
E_i=\mathrm{span}\{v_j^{(l)}(0), \; 1\leq j\leq k, 0,\leq l\leq i-1\}.
\end{equation}
Therefore $E_i$  is the subspace of $\bbR^n$ defined by the columns of the matrices $A^jB$ for $0\leq j\leq i-1$ and has dimension $k_i$. Choose coordinates on $\bbR^n$ adapted to this filtration, i.e., associated with a basis $\{e_{1},\ldots,e_{n}\}$ of $\R^{n}$ such that $\mathrm{span}\{e_1,\ldots,e_{k_i}\}=E_i$. In these coordinates $V(t)$ has a peculiar structure, namely
\begin{equation}\label{eq:V}
V(t)=\left(
\begin{array}{c}
\widehat{v}_1\\
t\widehat{v}_2\\
\vdots\\
t^{m-1}\widehat{v}_1
\end{array}
\right)+\left(\begin{array}{c}
t\widehat{w}_1\\
t^2\widehat{w}_2\\
\vdots\\
t^{m}\widehat{w}_1
\end{array}\right)+\left(\begin{array}{c}
O(t^2)\\
O(t^3)\\
\vdots\\
O(t^{m+1})
\end{array}\right),
\end{equation}
where $\widehat{v}_i$ and $\widehat{w}_i$ are $(k_i-k_{i-1})\times k$ constant matrices and every $\widehat{v}_i$ has maximal rank. Let $\widehat{V}(t)$ and $\widehat{W}(t)$ be the first and second principal parts of $V(t)$, then the Taylor series of the matrix $\Gamma_t$ can be found as
$$\Gamma_t=\int_0^t V(\tau)V(\tau)^*d\tau=\int_0^t \widehat{V}(\tau)\widehat{V}(\tau)^*+\left(\widehat{V}(\tau)\widehat{W}(\tau)^*+\widehat{W}(\tau)\widehat{V}(\tau)^*\right)d\tau+r(t),$$
where $r(t)$ is a remainder term. In components we write $\Gamma_t$ as a $m\times m$ block matrix, whose block $\Gamma_{ij}(t)$ is the $(k_i-k_{i-1})\times (k_j-k_{j-1})$ matrix with Taylor expansion
\begin{equation*}
\begin{split}
\Gamma_{ij}(t)=&\left(\frac{\widehat{v}_i\widehat{v}_j^*}{i+j-1}\right)t^{i+j-1}+\left(\frac{\widehat{v}_i\widehat{w}_j^*+\widehat{w}_i\widehat{v}_j^*}{i+j}\right)t^{i+j}+O\left(t^{i+j+1}\right)\\
&=\calX_{ij}t^{i+j-1}+\calY_{ij}t^{i+j}+O\left(t^{i+j+1}\right),
\end{split}
\end{equation*}
where $\calX$ and $\calY$ are $m\times m$ block matrices implicitly defined by this formula. Moreover $\calX$ is invertible. Let $J_{\sqrt{t}}$ be the $n\times n$ diagonal matrix whose $j$-th element is equal to $\sqrt{t}^{2i-1}$ for $k_{i-1}<j\leq k_i$, then
\begin{equation}\label{eq:GX}
\Gamma_t=J_{\sqrt{t}}\left(\calX+t\calY +O(t^2)\right)J_{\sqrt{t}}.
\end{equation}
As expected the determinant of $\Gamma_t$ is of order $t^\calN$ and its Taylor expansion is computed in terms of $\calX$ and $\calY$ as
$$\det \Gamma_t= t^\calN \det( \calX)(1+ \mathrm{tr}(\calX^{-1}\calY)t+o(t)).$$

Now we are ready to find the determinant of $D_t$. This follows from identity $D_t=e^{tA}\Gamma_t e^{tA^*}=-\Gamma_{-t}$. 
Indeed on the one hand the determinant of $D_t$ is given by
\begin{equation}
\begin{split}
\det D_t&=\det(e^{tA}\Gamma_t e^{tA^*})=\det(e^{2tA})\det \Gamma_t\\
%&=[1+2t\mathrm{tr}A+o(t)]t^N \det (\calX)[1+t \mathrm{tr}(\calX^{-1}\calY)+o(t)]\\
&=\det( \calX) t^\calN \left[1+\left(2\mathrm{tr}A+ \mathrm{tr}(\calX^{-1}\calY)\right)t+o(t)\right].
\end{split}
\label{eq:Dt1}
\end{equation}
On the other hand
\begin{equation}
\begin{split}
\det D_t&=\det(-\Gamma_{-t})=(-1)^n(-t)^\calN \det (\calX)[1-t\; \mathrm{tr}(\calX^{-1}\calY)+o(t)]\\
&=\det( \calX) t^\calN\left[1- \mathrm{tr}\left(\calX^{-1}\calY\right)t+o(t)\right].
\end{split}
\label{eq:Dt2}
\end{equation}
The last identity follows since $n$ and $\calN$ have the same parity, indeed
$$\calN=\sum_{i=1}^m (2i-1)k_i=2M-n,$$
where $M=\sum_{i=1}^m ik_i$.
Comparing equations \eqref{eq:Dt1} and \eqref{eq:Dt2} we find $\mathrm{tr}(\calX^{-1}\calY)=-\mathrm{tr}A$, that means
$$\det D_t=\det( \calX) t^\calN (1+(\mathrm{tr}A) t+o(t)).$$
It follows that in formula \eqref{eq:4} we have $c=-\mathrm{tr}A$ and $c_{0}=\det(\calX)>0$. 
This allows us to conclude that the asymptotics of the fundamental solution in $x=y\in \ker A$ for small time is
\begin{equation}\label{eq:sqrt}
\begin{split}
p(t,x,x)&=\frac{1}{(2\pi)^{n/2}\sqrt{\det D_t}}\\
&=\frac{ t^{-\calN/2}}{(2\pi)^{n/2}\sqrt{c_0}}e^{-\frac{\tr A}{2} t}e^{\frac{1}{2}\sum_{i=0}^h (-1)^{i}\tr\calQ^{(i)}\frac{t^{i+2}}{(i+1)(i+2)}+O(t^{h+3})}\\
&=\frac{ t^{-\calN/2}}{(2\pi)^{n/2}\sqrt{c_0}}\left(1+\sum_{i=1}^h a_i t^i + O(t^{h+1})\right),
\end{split}
\end{equation}
where the coefficients $a_i$ can be explicitly computed from the expansion of the exponential. It follows that every $a_i$ is polynomial in the components of $\tr A$ and $\tr \calQ^{(j)}$ for $j\leq i-2$ of order $i$ and does not depend on $x$. In particular the first coefficients are computed as follows
\begin{gather}
a_{1}=-\frac{\tr A}{2},\qquad 
a_{2}=\frac{(\tr A)^{2}}{8}+\frac{\tr \calQ^{(0)}}{4},\\[0,2cm]
a_{3}=-\frac{\tr \calQ^{(1)}}{12}-\frac{\tr A \,\tr \calQ^{(0)}}{8} -\frac{1}{48} (\tr A)^{3}.
\end{gather}

\subsection{Proof of Theorem \ref{th:cost}}\label{s:proof1b}
The proof of Theorem \ref{th:cost} is now an easy consequence of the analysis done untill here.

Indeed by Remark \ref{rk:cost function}, the geodesic that connects the point $x$ to $y$ in time $t$ has starting associated covector  $p_0=\Gamma_t^{-1}\left(e^{-t A}y-x\right)$, therefore the minimal cost to go from $x$ to $y$ in time $t$ is
$$S_t(x,y)=\frac{1}{2}\left(e^{-t A}y-x\right)^*\Gamma_t^{-1}\left(e^{-t A}y-x\right)=\frac{1}{2}(y-e^{tA}x)^*D_t^{-1}(y-e^{tA}x).$$
This is exactly the quantity appearing at the exponential of the heat kernel, that therefore can be written in terms of the optimal cost as
$$p(t,x,y)=\frac{e^{-S_t(x,y)}}{(2\pi)^{n/2}\sqrt{\det{D_t}}}.$$ 
The conclusion of the theorem is then a consequence of the Taylor expansion of $\frac{1}{\sqrt{\det{D_t}}}$, given in \eqref{eq:sqrt}.

%%%%%%%%%%%%%%%%%%%%%%%%%%%%%%%%%%%%%%%%%%%%%%%%%%%%%%%%%%%%%%%%%%%%%%%%%%%%%%%%%%%%%%%%%%%%%%%%%%%%%%%%%%%%%%%%%%%%%%
\section{Small time asymptotics out of the equilibrium} \label{s:proof2}
In this section we prove Theorem \ref{th:nonKer}, concerning the small time asymptotics of the fundamental solution at a point, $x_0$, where the drift field is not zero.

By a translation of the origin, we can assume that $x_0=0$. This produces a no more linear, but affine drift field, i.e., $X_0(x)=Ax+\alpha$, where $\alpha:=A x_0$ is a column vector equal to the value of the drift at $x_0$. Then we can study the original asymptotics at $x_0$, through the asymptotics at the origin of the heat kernel of the linear pde, where to the drift field we add the constant value $\alpha$. As shown in the Appendix its fundamental solution is
$$p(t,x,y)=\frac{e^{\varphi(t,x,y)}}{(2\pi)^{n/2}\sqrt{\det D_t}},$$
where
$$\varphi(t,x,y)=-\frac{1}{2}\left(y-e^{tA}\left(x+\int_0^t e^{-sA}ds\alpha\right)\right)^*D_t^{-1}\left(y-e^{tA}\left(x+\int_0^t e^{-sA}ds\alpha\right)\right)$$
and $D_t=e^{tA}\Gamma_t e^{tA^*}$ is the same covariance matrix as in \eqref{eq:Dt}. The original asymptotics of the fundamental solution at $x_0$ is given by the asymptotics of $(\det D_t)^{-1/2}$, found in the previous section, and the asymptotics of $\varphi$ in $x=y=0$, i.e.,
$$\varphi(t,0,0)=-\frac{1}{2}\alpha^*\left(\int_0^t e^{-sA}ds\right)^*\Gamma_t^{-1}\left(\int_0^t e^{-sA}ds\right)\alpha.$$

Let $E_i$ be the subspaces defined in \eqref{eq:E2}. If $\alpha= Ax_0\in E_i\setminus E_{i-1}$, then $A^j \alpha\in E_{i+j}$ (actually it could possibly live in the previous subspaces, but not in a bigger one). Moreover $\int_0^t e^{-sA}ds=\sum_{i=1}^m -\frac{(-t)^i}{i!} A^{i-1}+O(t^{m+1})$. Therefore in coordinates adapted to the filtration $\{E_j\}_j$, the column vector $\int_0^t e^{-sA}ds\,\alpha$ can be written in $m$ blocks of height $d_j$, as
$$\int_0^t e^{-sA}ds\,\alpha=\left(\begin{array}{c}
	t\alpha_1 \\
	t \alpha_2 \\
	\vdots \\
	t \alpha_i  \\
	t^2 \alpha_{i+1} \\
	\vdots \\
	t^{m-i+1} \alpha_m 
\end{array}\right)+\left(\begin{array}{c}
	O(t^2) \\
	O(t^2) \\
	\vdots \\
	O(t^2) \\
	O(t^3) \\
	\vdots \\
	O(t^{m-i+2})
\end{array}\right),
$$
where $\alpha_j$ is a vector of length $d_j$ for every $1\leq j\leq m$ and $\alpha_i$ is not zero. The matrix $\Gamma_t^{-1}$ can be written as a product of
$$\Gamma_t^{-1}=J_{1/\sqrt{t}} (\calX^{-1}+O(t))J_{1/\sqrt{t}},$$
where $J_{1/\sqrt{t}}$ and $\calX$ are the matrices introduced in \eqref{eq:GX}. Notice that the multiplication of $J_{1/\sqrt{t}}$ with $\int_0^t e^{-sA}ds\,\alpha$ gives the vector
$$
J_{1/\sqrt{t}}\int_0^t e^{-sA}ds\,\alpha=\left(\begin{array}{c}
	\sqrt{t}\alpha_1  \\
	\frac{1}{\sqrt{t}} \alpha_2 \\
	\vdots \\
	\frac{1}{\sqrt{t}^{2i-3}} \alpha_i \\
	\frac{1}{\sqrt{t}^{2i-3}} \alpha_{i+1}  \\
	\vdots \\
	\frac{1}{\sqrt{t}^{2i-3}} \alpha_m 
\end{array}\right)+\left(\begin{array}{c}
	O(t\sqrt{t}) \\
	O(\sqrt{t}) \\
	\vdots \\
	O(\frac{1}{\sqrt{t}^{2i-1}}) \\
	O(\frac{1}{\sqrt{t}^{2i-1}}) \\
	\vdots \\
	O(\frac{1}{\sqrt{t}^{2i-1}})
\end{array}\right).
$$
From here we see immediately that, since $\calX^{-1}$ has maximal rank and is positive definite, the scalar product
$$(J_{1/\sqrt{t}}\int_0^t e^{-sA}ds\,\alpha)^*\calX^{-1}(J_{1/\sqrt{t}}\int_0^t e^{-sA}ds\,\alpha)$$ 
has a simple pole of order $-(2i-3)$ in $0$. So if $\alpha\in E_i\setminus E_{i-1}$ with $i>1$, then $\varphi(t,0,0)$ blows up for small $t$ as $\frac{1}{t^{2i-3}}$, i.e., the heat kernel decreases with exponential velocity precisely as
$$p(t,x_0,x_0)= \frac{t^{-\calN/2}}{(2\pi)^{n/2}\sqrt{c_0}}\exp \left(-\frac{C_1+O(t)}{t^{2i-3}}\right) \qquad \mbox{ for } t\searrow 0,\; Ax_0\in E_i\setminus E_{i-1},$$
for a positive constant $C_1$. This proves the second part of Theorem \ref{th:nonKer}.

%Since we are interested in the first terms of the asymptotics, we can approximate all the matrices with their Taylor series
%$$\int_0^t e^{-sA}ds=tI+O(t^2)\quad \mbox{ and } \quad [\Gamma_t]^{-1}=I_{1/\sqrt{t}}(\calX^{-1}+O(t)) I_{1/\sqrt{t}}$$
%where $I_{\sqrt{t}}$ and $\calX$ are the same matrices introduced in the previous section, Eq. \eqref{eq:GX}. Hence 
%\begin{equation*}
%\begin{split}
%\varphi &(t,0,0)=-\frac{1}{2}\alpha^*\left(t I+O(t^2)\right)I_{1/\sqrt{t}}(\calX^{-1}+O(t)) I_{1/\sqrt{t}}\left(t I+O(t^2)\right)\alpha=\\
%-&\frac{1}{2}\alpha^*\left(\begin{array}{cccc}
%\sqrt{t} &   0        & 0            & \cdots \\
%0        & 1/\sqrt{t} & 0            & \cdots \\
%0        & 0          & 1/\sqrt{t}^3 & \cdots \\
%0        & 0          & 0            & \ddots
%\end{array}\right)(\calX^{-1}+O(t)) \left(\begin{array}{cccc}
%\sqrt{t} & 0 & 0 & \cdots \\
%0 & 1/\sqrt{t} & 0 & \cdots \\
%0 & 0 & 1/\sqrt{t}^3 & \cdots \\
%0 & 0 & 0 & \ddots
%\end{array}\right) \alpha,
%\end{split}
%\end{equation*}
%where the explicite diagonal matrices are made of $m\times m$ block, whose $ii$-block has dimension $d_i\times d_i$ and is the identity matrix multiplied by $\sqrt{t}^{3-2i}$.
%
%Since $\calX$ has maximal rank, we see immediately that if $\alpha:=X_0(x_0)\notin \mathrm{span}{B}$, then $\varphi(t,0,0)$ blows up for $t$ going to zero, hence the asymptotics of the fundamental solution, decreases with exponential velocity. Moreover if $\alpha$ is in the $i$-th layer for $2\leq i\leq m$, i.e. $\alpha\in E_i\setminus E_{i-1}$, then it decreases as $\exp\left(-\frac{1}{t^{2i-3}}\right)$.

Let us consider now the case $\alpha\in\mathrm{span}B$. With this assumption the function $\varphi(t,0,0)$ is smooth in $t=0$ and we want to find its exact value at the first order.

With a change of variables we can assume that the matrix $B$ is the identity matrix in the first $k$ rows and zero on the last $n-k$ rows. Moreover let $y\in \bbR^k$ be such that $Ax_0=By$. We claim that  
$$\varphi(t,0,0)=-t\frac{|y|^{2}}{2}+o(t)=-t\frac{|Ax_0|^{2}}{2}+o(t).$$ 
We can write the function $\varphi$ as
$$\varphi(t,0,0)=-\frac{1}{2}y^*\left(\int_0^t e^{-sA}Bds\right)^*\Gamma_t^{-1}\left(\int_0^t e^{-sA}Bds\right)y$$
and we need to characterize $e^{-sA}B$. Using the results of the previous section, in coordinates adapted to the filtration we can write 
$$\int_0^t e^{-sA}Bds=\left(
	\begin{array}{c}
		t\widehat{v}_1 \\
		t^2\frac{\widehat{v}_2}{2} \\
		\vdots \\
		t^{m}\frac{\widehat{v}_m}{m}
	\end{array}
	\right)
+\left(
	\begin{array}{c}
		O(t^2) \\
		O(t^3) \\
		\vdots \\
		O(t^{m+1})
	\end{array}\right),$$
where the $\widehat{v}_i$ are defined in \eqref{eq:V}. They are $d_i\times k$ matrix of maximal rank and $\widehat{v}_1$ is the $k\times k$ identity matrix. We can highlight the dependence on $t$ by multiplying on the left by the diagonal matrix $K_t$ with $j$-th entry equal to $t^{i}$, for $k_{i-1}<j\leq k_i$. Also the matrix $\Gamma_t^{-1}$ can be written as a product (see \eqref{eq:GX}), then 
$$\int_0^t e^{-sA}Bds= K_t \left(\left(\begin{array}{c}
\mathbb{I}_k \\
C
\end{array}\right)+O(t)\right)\quad \mbox{ and } \quad \Gamma_t^{-1}=J_{1/\sqrt{t}}\left(\calX^{-1}+O(t)\right)J_{1/\sqrt{t}},$$
where $C$ is the $(n-k)\times k$ matrix composed by the last $n-k$ rows of the principal part of $\int_0^t e^{-sA}Bds$. Since $K_t \cdot J_{1/\sqrt{t}}=\sqrt{t}\, \mathbb{I}_n$, then
\begin{equation}
\varphi(t,0,0)=-\frac{t}{2}y^*\left(\begin{array}{c}
\mathbb{I}_k \\
C
\end{array}\right)^*\calX^{-1}\left(\begin{array}{c}
\mathbb{I}_k \\
C
\end{array}\right)y+O(t^2).
\label{eq:varphiB}
\end{equation}

Notice that $[\calX]_{11}=\mathbb{I}_k$ and $[\calX]_{i1}=\frac{\widehat{v}_i}{i}=[\calX]_{1i}^*,$ for every $2\leq i\leq m$, hence $\calX$ can be written as a block matrix
$$\calX=\left(\begin{array}{cc}
 \mathbb{I}_k & C^* \\
 C & E 
\end{array}\right),$$
where $E$ is a $(n-k)\times(n-k)$ matrix.

\begin{lemma}
The inverse matrix $\calX^{-1}$ is the block matrix 
$$\calX^{-1}=\begin{pmatrix}
[\calX^{-1}]_{11} & [\calX^{-1}]_{12} \\
[\calX^{-1}]^{*}_{12} & [\calX^{-1}]_{22} 
\end{pmatrix},$$
where
\begin{gather}
[\calX^{-1}]_{11}=\mathbb{I}_k+C^*(E-CC^*)^{-1}C, \\
[\calX^{-1}]_{12}=-C^*(E-CC^*)^{-1}, \qquad [\calX^{-1}]_{22}=(E-CC^*)^{-1}.
\end{gather}
\end{lemma}
\begin{proof}
This is the general expression of the inverse of a block matrix, provided $[\calX]_{11}$ and $E-CC^*$ are not singular. $[\calX]_{11}$ is the identity matrix. Let us show that $E-CC^*$ is not singular. Assume $x\in \bbR^{n-k}$ satisfies $(E-CC^*)x=0$. Then the column vector (of dimension $n$) equal to $(-(C^*x)^*,x^*)^*$ is in the kernel of $\calX$. Therefore $x=0$, since $\calX$ is not singular.
\end{proof}

Applying the Lemma to Eq.\ \eqref{eq:varphiB} we find that
\begin{equation*}
\begin{split}
\varphi(t,0,0)&=-\frac{t}{2}y^*\left([\calX^{-1}]_{11}+([\calX^{-1}]_{12}C)^*+[\calX^{-1}]_{12}C+C^*[\calX^{-1}]_{22}C\right)y+O(t^2)\\
&=-\frac{t}{2}y^*\,\mathbb{I}_k\, y+O(t^2)=-\frac{| Ax_0|^2}{2}t+O(t^2).
\end{split}
\end{equation*}
Taking into account the asymptotics of $(\det D_t)^{-1/2}$ found in the previous section, this completes the proof of Theorem \ref{th:nonKer}.

\appendix
\section{On the fundamental solution}
In the following section we derive the fundamental solution of the linear partial differential equation
\begin{equation}
\frac{\partial\varphi}{\partial t}-\sum_{j=1}^n(\alpha+Ax)_j\frac{\partial\varphi}{\partial x_j} -\frac{1}{2} \sum_{j, h=1}^n (BB^*)_{jh}\frac{\partial^2\varphi}{\partial x_j\partial x_h}=0, \quad \mbox{ for } \varphi\in C^\infty_0(\bbR\times\bbR^n),
\label{eq:pdelin}
\end{equation}
where $\alpha$ is a constant column vector of dimension $n$, $A$ is a $n\times n$ real matrix, which represents the linear part of a \emph{drift} field, and $B$ is a $n\times k$ real matrix, with $1\leq k\leq n$, that generates the diffusion coefficient.

An operator $\partial_t-L$ is hypoelliptic if for every open set $U\in \bbR^n$ and for every function $\varphi$ such that $(\partial_t-L)\varphi\in C^\infty(U)$, then also $\varphi\in C^\infty(U)$. H\"{o}rmander proved in \cite{article:Hormander} that if Kalman's condition on the matrices $A,B$ holds, that is
\begin{equation}
\mathrm{rk}[B,AB,A^2B,\ldots,A^{n-1}B]=n,
\label{eq:kalman}
\end{equation}
then the operator in \eqref{eq:pdelin} is hypoelliptic. Moreover in this case the operator admits a well defined fundamental solution, $p(t,x,y)$, that is given by the probability density  of the solution $\xi_t$ of the associated stochastic differential equation
\begin{equation}
\left\{\begin{array}{l}
d\xi_t=(\alpha+A\xi_t)dt+B dw(t)\\
\xi_0=x
\end{array}\right.
\label{eq:sde}
\end{equation}
where $w(t)=(w_1(t),\ldots,w_k(t))$ is a $k$-dimensional Brownian motion. In other words, $p(t,x,y)$ is the $C^\infty(\bbR^+\times\bbR^n\times\bbR^n)$ function such that for every Borel set $U\in\bbR^n$ the probability of $\xi_t$ to be in $U$ at time $t$ is given by
$$P\left[\xi_t\in U|\xi_0=x\right]=\int_U p(t,x,y)dy.$$
Here $\bbR^+$ denotes the strictly positive real line. The stochastic process solution of \eqref{eq:sde} is equal to
$$\xi_t=e^{tA}\xi_0+e^{tA}\int_0^t e^{-s A}ds \alpha+e^{tA}\int_0^t e^{-sA} B dw(s),$$
as it can be readily verified with a derivation. In particular, if the initial value $\xi_0$ is Gaussian distributed, then $\xi_t$ is a Gaussian process, since by definition of solution $\xi_0$ and $w(t)$ are independent. The initial condition is a Dirac delta centered in $x$, so it is a degenerate Gaussian with mean $x$ and vanishing covariance matrix, therefore to find the distribution of $\xi_t$ it is enough to determine its mean value and covariance matrix. To this end let us recall It\^{o}'s differential formula: let $\psi(t,x)$ be a smooth function, then for every $0\leq s<t$
\begin{equation}
\begin{split}
\psi(t,\xi_t)-\psi(s,\xi_s)&=\int_s^t\frac{\partial \psi(\tau,\xi_\tau)}{\partial t}d\tau\\
+&\int_s^t\sum_{j=1}^n (\alpha+A\xi_\tau)_j\frac{\partial \psi(\tau,\xi_\tau)}{\partial x_j}+\frac{1}{2}\sum_{j,h=1}^n (BB^*)_{jh}\frac{\partial^2 \psi(\tau,\xi_\tau)}{\partial x_j\partial x_h}d\tau\\
+&\int_s^t \sum_{i=1}^k \sum_{j=1}^n B_{ji}\frac{\partial \psi(\tau,\xi_\tau)}{\partial x_j}dw_i(\tau).
\end{split}
\end{equation}
To find the mean value of $\xi_t$ let us fix the auxiliary function $\psi:=x_j$ for $1\leq j\leq n$, then
$$\xi_j(t)-\xi_j(0)=\int_0^t(\alpha+A\xi(\tau))_j d\tau+\int_0^t\sum_{i=1}^k B_{ji}dw_i(\tau).$$
Let $m_j(t)$ be the $j$-th component of the mean value of $\xi_t$. By the previous formula, since the integral of the Brownian motion has zero mean value, we find that $m(t)$ satisfies the differential equation $\dot{m}_j(t)=(\alpha+Am(t))_j$ with initial condition $m(0)=x$. Therefore
$$\mathrm{E}[\,\xi_t\,|\,\xi_0\,]=e^{tA}\left(x+\int_0^t e^{-sA}ds \alpha\right).$$
To compute the covariance matrix let us choose $\psi=(x_j-m_j(t))(x_h-m_h(t))$, then by It\^{o} differential formula it holds
\begin{equation}
\begin{split}
\psi(t,\xi(t))&-\psi(0,\xi(0))=\int_0^t-\dot{m}_j(\tau)(\xi_h(\tau)-m_h(\tau))-\dot{m}_h(\tau)(\xi_j(\tau)-m_j(\tau))d\tau\\
+&\int_0^t(\alpha+A\xi)_j(\xi_h(\tau)-m_h(\tau))+(\alpha+A\xi)_h(\xi_j(\tau)-m_j(\tau))+(BB^*)_{jh}d\tau\\
+&\int_0^t \sum_{i=1}^k B_{ji}(\xi_h(\tau)-m_h(\tau))+B_{hi}(\xi_j(\tau)-m_j(\tau))dw_i(\tau).
\end{split}
\end{equation}
Let us take the expectation value and denote by $\rho_{jh}(t)$ the $jh$-component of the covariance matrix. Recall that $\mathrm{E}[(\xi_j(t)-m_j(t))|\xi(0)]=0$. To evaluate $$\mathrm{E}\left[(A\xi(\tau))_j(\xi_h-m_h(\tau))|\xi(0)\right]$$ we rewrite it as
$$\mathrm{E}\left[(A(\xi(\tau)-m(\tau)))_j(\xi_h-m_h(\tau))+(Am(\tau))_j(\xi_h-m_h(\tau))|\xi(0)\right]=(A\rho(\tau))_{jh}+0.$$
Then $\rho(t)$ satisfies the differential equation $\dot{\rho}_{jh}(t)=(A\rho(t))_{jh}+(A\rho(t))_{hj}+(BB^*)_{jh}$ with vanishing initial value, whose solution is
\begin{equation}
D_t:=\mathrm{E}[(\xi_j-m_j(t))(\xi_h-m_h(t))|\xi(0)]=e^{tA}\int_{0}^t e^{-\tau A}BB^* e^{-\tau A^*}d\tau e^{tA^*}.
\label{eq:Dt4}
\end{equation}
By Kalman's condition \eqref{eq:kalman}, the matrix $D_t$ is invertible for every $t>0$. Therefore we can conclude that the $C^\infty$ fundamental solution of equation \eqref{eq:pdelin} is given by the non-degenerate Gaussian
$$p(t,x,y)=\frac{e^{\varphi(t,x,y)}}{(2\pi)^{n/2}\sqrt{\det D_t}},$$
where
$$\varphi(t,x,y)=-\frac{1}{2}\left(y-e^{tA}\left(x+\int_0^t e^{-sA}ds\alpha\right)\right)^*D_t^{-1}\left(y-e^{tA}\left(x+\int_0^t e^{-sA}ds\alpha\right)\right).$$
In the case when $\alpha=0$ the formula for $p(t,x,y)$ reduces to
$$p(t,x,y)=\frac{e^{-\frac{1}{2}(y-e^{tA}x)^*D_t^{-1}(y-e^{tA}x)}}{(2\pi)^{n/2}\sqrt{\det{D_t}}} \qquad\qquad \mbox{ for } t>0,  x,y\in\bbR^n.$$

\medskip

\subsection*{Acknowledgments} This research has been supported by the European Research Council, ERC StG 2009 ``GeCoMethods'', contract number 239748 and by the ANR project SRGI ``Sub-Riemannian Geometry and Interactions''. %number xxxxx. 

\bibliography{bibliografia}
\bibliographystyle{plain} 
\end{document}